\documentclass[preprint,11pt]{elsarticle}

\usepackage{amssymb}
\usepackage{amsthm}

\usepackage{latexsym, amssymb, bm, verbatim, amsmath, amsthm}

\usepackage{fancyhdr}
\usepackage{bbm}
\usepackage{amsthm}
\usepackage{graphicx,graphics}

\DeclareMathOperator{\stem}{stem}
\DeclareMathOperator{\Dom}{dom}
\DeclareMathOperator{\ran}{ran}
\DeclareMathOperator{\NUF}{NUF}

\DeclareMathOperator{\next}{next}
\DeclareMathOperator{\Succ}{Succ}

\def\PP{{\mathbb P}}
\def\QQ{{\mathbb Q}}

\newcommand{\ra}{\rangle}

\newcommand{\al}{\alpha}
\newcommand{\be}{\beta}

\newcommand{\ka}{\kappa}

\renewcommand{\phi}{\varphi}

\newcommand{\bbM}{\mathbb{M}}

\newcommand{\bbP}{\mathbb{P}}

\newcommand{\st}{\mid} 

\DeclareMathOperator{\add}{add}

\DeclareMathOperator{\cof}{cof}

\DeclareMathOperator{\cov}{cov}

\DeclareMathOperator{\dom}{dom}

\DeclareMathOperator{\non}{non}

\newcommand{\calD}{\mathcal{D}}

\newcommand{\calF}{\mathcal{F}}

\newcommand{\calM}{\mathcal{M}}

\newcommand{\calU}{\mathcal{U}}

\newcommand{\fra}{\mathfrak{a}}
\newcommand{\frb}{\mathfrak{b}}

\newcommand{\frd}{\mathfrak{d}}

\newcommand{\frh}{\mathfrak{h}}

\newcommand{\frp}{\mathfrak{p}}
\newcommand{\frr}{\mathfrak{r}}
\newcommand{\frs}{\mathfrak{s}}

\newcommand{\frt}{\mathfrak{t}}
\newcommand{\fru}{\mathfrak{u}}

\newtheorem{thm}{Theorem}
\newtheorem{defn}[thm]{Definition}

\newtheorem{lem}[thm]{Lemma}
\newtheorem{coroll}[thm]{Corollary}
\newtheorem{prop}[thm]{Proposition}

\journal{Annals of Pure and Applied Logic}

\begin{document}

\begin{frontmatter}



\title{Cardinal characteristics at $\ka$ in a small $\fru(\ka)$ model}


\author[label1]{A. D. Brooke-Taylor}
\author[label2]{V. Fischer\corref{cor1}}
\ead{vera.fischer@univie.ac.at}
\author[label2]{S. D. Friedman}
\author[label2]{D. C. Montoya}
\address[label1]{School of Mathematics, University of Bristol, University Walk,
Bristol, BS8 1TW, UK.}
\address[label2]{Kurt G\"odel Research Center, W\"ahringer Strasse 25, 1090 Vienna, Austria.}
\cortext[cor1]{Corresponding author}

\begin{abstract}
We provide a model where $\mathfrak{u}(\ka)= \ka^+ < 2^\ka$ for a supercompact cardinal $\ka$. \cite{GaS:PCC} provides a sketch of how to obtain such a model by modifying the construction in \cite{DzS:UGSSC}.
We provide here a complete proof using a different modification of \cite{DzS:UGSSC} and further study the values of other natural generalizations of classical cardinal characteristics in our model.
For this purpose we generalize some standard facts that hold in the countable case as well as some classical forcing notions and their properties.
\end{abstract}


\end{frontmatter}



\section{Introduction}

Cardinal invariants on the Baire space $\omega^\omega$ have been widely studied
and understood. Since 1995 with the Cummings-Shelah paper \cite{JC:SS}, the study of the generalization of these cardinal notions
to the context of uncountable cardinals and their interactions has been
developing. By now, there is a wide literature on this topic. Some key references (at least for the purposes of this paper) are
\cite{BHZ:MFN}, \cite{JC:SS} and \cite{SZ}.

In \cite{DzS:UGSSC} D\v{z}amonja and Shelah construct a model with a universal graph at
the successor of a strong limit singular cardinal of countable cofinality.
A variant of this model, as pointed out by Garti and Shelah in \cite{GaS:PCC}, witnesses
the consistency of $\mathfrak{u}(\ka)= \ka^+ < 2^\ka$ (Here $\mathfrak{u}(\ka)= \min\{ \lvert \mathcal{B}\lvert : \mathcal{B} $ is an base for a uniform ultrafilter on $\ka\} $). See also \cite{ABT:SukLck}.

Here we present a modification of the forcing construction used by D\v{z}amonja and Shelah, which allows us to prove that if $\ka$ is a supercompact cardinal and $\ka < \ka^\ast$ with $\ka^*$ regular, then there is a generic extension of the universe in which cardinals have not been changed and $\mathfrak{u}(\ka)=\ka^\ast$. The idea of our construction originates in \cite{FL:CCU} and states that if after the iteration $\ka$
is still supercompact (which can be guaranteed by using the Laver preparation)
and we take a normal measure $\mathcal{U}$ on $\kappa$ in the final extension,
then there is a set of ordinals of order type $\ka^\ast$ such that the restrictions
of $\mathcal{U}$ to the corresponding intermediate extensions coincide with ultrafilters
which have been added generically (see Lemma~\ref{ultrafilter}).
In addition, to obtain $\mathfrak{u}(\kappa) =\kappa^\ast$ we further
ensure that each of these restricted ultrafilters contains a Mathias generic for its smaller restrictions, yielding then an ultrafilter generated by these $\kappa^\ast$-many Mathias generics.

Moreover our construction allows us to decide the values of many of the higher analogues of the known classical cardinal characteristics of the continuum, as we can interleave arbitrary $\ka$-directed closed
posets cofinally in the iteration.  The detailed construction of our model is presented in Section 3, while our applications appear in Section 4.

Thus our main result, states the following:

\begin{thm}
Suppose $\kappa$ is a supercompact cardinal, $\ka^\ast$ is a regular cardinal with
$\kappa <\kappa^{\ast} \leq \Gamma$ and $\Gamma$ satisfies $\Gamma^\kappa= \Gamma$. Then there is forcing extension in which cardinals have not been changed satisfying:
\begin{align*}
\ka^*&=\fru(\ka)=\frb(\ka)=\frd(\ka)=\fra(\ka)=\frs(\ka)=\frr(\ka)=\cov(\calM_\ka)\\
&=\add(\calM_\ka) =\non(\calM_\ka)=\cof(\calM_\ka) \;\hbox{and}\;2^\kappa=\Gamma.
\end{align*}
If in addition $\gamma<\kappa^*\rightarrow\gamma^{<\kappa}<\kappa^*$, then we can also provide that $\mathfrak{i}(\ka)= \ka^\ast$. If in addition $(\Gamma)^{<\kappa^*}\leq\Gamma$ then we can also provide that $\frp(\ka)=\frt(\ka)=\frh(\ka)=\kappa^*$.
\end{thm}

In addition, we establish some of the natural inequalities between the invariants (in the countable case these are well known).

\section{Preliminaries}

Let $\kappa$ be a supercompact cardinal. Recall that this means that for all
$\lambda \geq \kappa$ there is an elementary embedding $j: V \to M$ with critical
point $\kappa$, $j(\kappa)> \lambda$ and $M ^\lambda \subseteq M$.

One of the main properties of supercompact cardinals that will be used throughout
the paper is the existence of the well-known Laver preparation,
which makes the supercompactness of $\kappa$
indestructible by subsequent forcing
with $\kappa$-directed-closed partial orders.

\begin{thm}[Laver, \cite{Lav:prep}]
If $\kappa$ is supercompact, then there exists a $\kappa$-cc partial ordering $S_\kappa$ of size $\kappa$ such that
 in $V^{S_\kappa}$, $\kappa$ is supercompact and remains supercompact after forcing with any $\kappa$-directed closed
 partial order.
\end{thm}

The main lemma used to obtain this theorem is the statement
that for any supercompact cardinal $\kappa$ there exists
a \emph{Laver diamond}.
That is, there is a function $h: \kappa \to V_\kappa$ such that for
every set $x$ and every cardinal $\lambda$, there is an elementary embedding
$j: V \to M$ with critical point $\kappa$, $j(\kappa)> \lambda$, $M^\lambda \subseteq M$ and $j(h)(\kappa)=x$.\\
Given such a function, the Laver preparation $S_\kappa$ is given explicitly as
a reverse Easton iteration
$(S_\alpha, \dot{R}_\beta: \alpha \leq \kappa, \beta < \kappa)$,
defined alongside a sequence of cardinals
$(\lambda_\alpha: \alpha < \kappa)$ by induction on $\alpha < \kappa$ as follows.

\begin{itemize}
\item If $\alpha$ is a cardinal and $h(\alpha)= (\dot P, \lambda)$, where $\lambda$ is a cardinal, $\dot{P}$ is an $S_\alpha$ name for a $<\alpha$-directed closed forcing,
and for all $\beta < \alpha$, $\lambda_\beta < \alpha$, we let $\dot{R}_\alpha := \dot P$ and $\lambda_\alpha= \lambda$.
\item Otherwise, we
let $\dot{R}_\alpha$ be the canonical name for the trivial forcing and
$\lambda_\alpha= \sup_{\beta<\alpha} \lambda_\beta$.
\end{itemize}

One of the main forcing notions we will use is the following:

\begin{defn}[Generalized Mathias Forcing]\label{def_gmathias}
Let $\ka$ be a measurable cardinal, and let
$\mathcal{F}$ be a $\kappa$-complete filter on $\kappa$.
The Generalized Mathias Forcing $\bbM^\ka_\calF$ has,
as its set of conditions,
$\{ (s, A): s \in [\kappa]^{<\kappa}\text{ and }A \in \mathcal{F} \}$,
and the ordering given by
$(t,B) \leq (s,A)\text{ if and only if }
t \supseteq s, B\subseteq A\text{ and }t \setminus s \subseteq A$.
We denote by $\mathbbm{1}_\mathcal{F}$ the maximum element of $\bbM^\ka_\mathcal{F}$,
that is $\mathbbm{1}_\mathcal{F}=(\emptyset, \ka)$.
\end{defn}

In our main forcing iteration construction we work exclusively with generalized Mathias posets $\mathbb{M}^\kappa_{\mathcal{U}}$, where $\mathcal{U}$
is a $\kappa$-complete ultrafilter. In our applications however, we will be working with arbitrary $\kappa$-complete filters.

\begin{defn}\label{kacentred}
A partial order $\bbP$ is:
\begin{itemize}
 \item
 \emph{$\ka$-centered} if there is a partition
$\{\bbP_\al\st\al <\ka\}$ of $\bbP$
such that for each $\al< \ka$, every pair of conditions $p,q \in \bbP_\al$ has
a common extension in $\mathbb{P}_\alpha$;

\item {\emph{$\ka$-directed closed}} if for every directed set $D \subseteq \bbP$
of size $\lvert D \lvert < \ka$ there is a condition $p \in \bbP$ such that $p \leq q$ for all
$q \in D$.
\end{itemize}
\end{defn}

\section{The small $\mathfrak{u}(\kappa)$ model}


Let $\Gamma$ be such that $\Gamma^\kappa=\Gamma$. We will define
an iteration $\langle \PP_\alpha,\dot{\QQ}_\beta:\alpha\leq\Gamma^+,\beta<\Gamma^+\ra$ of
length $\Gamma^+$ recursively as follows:

{\emph{If $\alpha$ is an even ordinal}} (abbreviated $\alpha\in\hbox{EVEN}$), let $\NUF$ denote the
set of normal ultrafilters on $\kappa$ in $V^{\PP_\alpha}$. Then let $\QQ_\alpha$ be the poset with underlying set of conditions $\{\mathbbm{1}_{\mathbb{Q}_\al}\} \cup \{ \{\mathcal{U}\} \times \mathbb{M}^\kappa_{\mathcal{U}}: \mathcal{U} \in \NUF \}$ and extension relation stating that $q\leq p$ if and only if either $p=\mathbbm{1}_{\mathbb{Q}_\alpha}$, or there is $\mathcal{U}\in \NUF$ such that $p=(\mathcal{U},p_1)$, $q=(\mathcal{U}, q_1)$ and $q_1\leq_{\mathbb{M}^\kappa_{\mathcal{U}}}p_1$. {\emph{If $\alpha$ is an odd ordinal}} (abbreviated $\alpha\in\hbox{ODD}$), let $\dot{\mathbb{Q}}_\alpha$ be a $\mathbb{P}_\alpha$-name for a $\ka$-centered, $\kappa$-directed closed forcing notion of size at most $\Gamma$.

We define three different kinds of support for conditions $p \in \bbP_{\alpha}$, $\alpha< \Gamma^+$:
First we have the \emph{Ultrafilter Support} $\hbox{USupt}(p)$, that corresponds to the set of ordinals $\beta \in \dom(p) \cap \hbox{EVEN} $ such that
$p\restriction \beta\Vdash_{\PP_\beta}
p(\beta)\neq\mathbbm{1}_{\QQ_\beta}$. Then  the \emph{Essential
  Support} $\hbox{SSupt}(p)$, which consists of all
$\beta\in\Dom(p)\cap\hbox{EVEN}$ such that $\neg (p\restriction
\beta\Vdash_{\PP_\beta} p(\beta)\in \{
\check{\mathbbm{1}}_{\mathbb{Q}_\be}\} \cup \{ (\calU,
\mathbbm{1}_{\mathcal{U}}):\mathcal{U} \in \NUF\})$ (for the
definition of $\mathbbm{1}_{\mathcal{U}}$ see
Definition~\ref{def_gmathias}). Finally, the \emph{Directed Support}
$\hbox{RSupt}(p)$, consists of all $\beta\in \Dom(p) \cap\hbox{ODD}$
such that $\neg (p\restriction \beta\Vdash p(\beta)=
\mathbbm{1}_{\dot{\QQ}_\beta})$.

We require that the conditions in $\PP_{\Gamma^+}$ have support
bounded below $\Gamma^+$ and also that given $p \in \bbP_{\Gamma^+}$
if $\beta \in \hbox{USupt}(p)$ then for all $\alpha \in \beta \cap
\hbox{EVEN}$, $\alpha \in \hbox{USupt}(p)$. Finally we demand that
both $\hbox{SSupt}(p)$ and $\hbox{RSupt}(p)$ have size $< \ka$ and are
contained in $\sup(\hbox{USupt}(p))$, i.e. $\hbox{Supt}(p)$ (the
entire support of $p$) and $\hbox{USupt}(p)$ have the same supremum.

Now, we want to ensure that our iteration preserves cardinals. Let $\PP:=\PP_{\Gamma^+}$.

\begin{lem}
 $\mathbb{P}$ is $\kappa$-directed closed.
\end{lem}

\begin{proof}
We know that $\mathbb{M}^\kappa_\mathcal{U}$, as well as all iterands
$\mathbb{Q}_\alpha$ for
$\alpha \in \hbox{ODD}$, are $\ka$-directed closed forcings. Take $D=
\{p_\alpha: \alpha < \delta < \kappa\}$ a directed set of conditions
in $\mathbb{P}$. We want to define a common extension $p$ for all
elements in $D$. First define $\Dom (p) = \bigcup_{\alpha< \delta}
\Dom(p_\alpha)$. For $j \in \Dom(p)$ define $p(j)$ by induction on
$j$. We work in $V^{\bbP_j}$ and assume that $p\restriction
j \in  \bbP_j$.

We have the following cases:

\begin{itemize}
\item if $j$ is even and $j \notin \bigcup_{\alpha< \delta} \hbox{SSupt}(p_\alpha)$, then using
  compatibility we can find at most one normal ultrafilter
  $\mathcal{U}$ such that for some $\alpha< \delta$, $p_\alpha
  \restriction j \Vdash p_\alpha(j)=
  (\mathcal{U},\mathbbm{1}_{\mathcal{U}}) $. If there is such a
  $\mathcal{U}$ define $p(j)= (\mathcal{U},
  \mathbbm{1}_{\mathcal{U}})$, otherwise $p(j)=
  \mathbbm{1}_{\mathbb{Q}_j}$.

\item If $j$ is even and $j \in \hbox{SSupt}(p_\alpha)$ for some $\alpha < \delta$, then again
  using directedness it is possible to find a single ultrafilter $\mathcal{U}$ such that for
  $\alpha < \delta$ with $j \in \hbox{SSupt}(p_\alpha)$, $p_\alpha \restriction j
  \Vdash p_\alpha(j) \in  \mathcal{U} \times \mathbb{M}^\ka_{\mathcal{U}}$, and
  $\Vdash_{\mathbb{P}_j} \mathbb{M}^\ka_{\mathcal{U}}$ is
  $\ka$-directed closed. In the extension $V^{\mathbb{P}_j}$ we can find a
  condition $q$ such that $q \leq p_\alpha(j)$ for all $\alpha <
  \delta$.
  Define $p(j) = q$.

\item If $j$ is odd, use the fact that in the $\mathbb{P}_j$ extension $\mathbb{Q}_j$ is
  $\ka$-directed closed on the directed set $X_j= \{ p_\alpha(j): \alpha < \delta < \ka \}$
  to find $p(j)$ a condition stronger than all the ones in $X_j$.
\end{itemize}
\end{proof}

For any $p \in \bbP_\beta$ , $\beta < \Gamma^+$ let  $\PP_\beta \downarrow p$ denote the set
$\{q\in\PP_\beta: q\leq p\}$.

\begin{lem}\label{ccc}
Let $p\in\PP_{\Gamma^+}$ and let
$i=\sup\hbox{USupt}(p)=\sup\hbox{Supt}(p)$. Then $\PP_i\downarrow
(p\restriction i)$ is $\kappa^+$-cc and has a dense subset of size at
most $\Gamma$.
\end{lem}
\begin{proof}
It is enough to observe that  $\PP_i\downarrow (p\restriction i)$ is
basically a $<\kappa$-support iteration of $\ka$-centered,
$\ka$-directed closed forcings of size at most $\Gamma$.
Then the proof is a straightforward generalization of Lemma V.4.9 -- V.4.10 in \cite{Kun:ST}.
\end{proof}

\begin{lem}\label{anti}
Let $\{\mathcal{A}_\alpha\}_{\alpha < \Gamma}$ be maximal antichains
in $\bbP$ below $p\in\PP$. Let
$j^*=\sup\hbox{Supt}(p)$. Then there is $q\in\PP$ such that
$q\restriction j^*=p$,
$\hbox{Supt}(q)\backslash \hbox{Supt}(p)\subseteq\hbox{USupt}(q)$ and
for all $\alpha<\Gamma$, the set
$\mathcal{A}_\alpha\cap(\PP_{i^*}\downarrow q)$ is a maximal antichain in
$\PP_{i^*}\downarrow q$ (and hence in $\PP\downarrow q$), where $i^*=\sup\hbox{Supt}(q)$.
\end{lem}
\begin{proof}
Let $\bar{\PP}:=\PP_{j^*}\downarrow p$ and
let $w\in\bar{\PP}$. Then there is a condition $r$
extending both $w$ and an element of
$\mathcal{A}_0$
and we can find $p_1$ such that
$p_1\restriction j^*=p$ and $r\in \PP_{j_1}\downarrow p_1$,
where $j_1=\sup\hbox{Supt}(p_1)$. Since $\bar{\PP}$ has a dense subset
of size at most $\Gamma$,
in $\kappa^+$-steps we can find $q_0$ such that $q_0\restriction
j^*=p$ and every condition in $\bar{\PP}$ is compatible with an
element of $\mathcal{A}_0\cap (\PP_{j_0^*}\downarrow q_0)$,
where $j_0^*=\sup\hbox{Supt}(q_0)$.

Since we have only $\Gamma$ many antichains
$\{\mathcal{A}_\alpha\}_{\alpha < \Gamma}$
in $\Gamma$ steps we can obtain the desired condition $q$.
\end{proof}

\begin{coroll}\label{smallinit}
If $p \Vdash \dot{X} \subseteq \ka$ for some $\mathbb{P}$-name
$\dot{X}$, then there are $q\leq p$  and $j< \Gamma^+$ such that
$\dot{X}$ can be seen as a $\mathbb{P}_j \downarrow q$-name.
\end{coroll}
\begin{proof}
For each $\alpha<\kappa$ fix a maximal antichain $\mathcal{A}_\alpha$
of conditions below $p$ deciding if $\alpha$ belongs to $\dot{X}$.
Then, let $q$ be the
condition given by Lemma~\ref{anti}
and take $j:=\sup\hbox{Supt}(q)$. Then $q\leq p$ and $\dot{X}$ can be
seen as a $\PP_{\sup{\hbox{Supt}(q)}}\downarrow q$-name.
\end{proof}

\begin{coroll}
Let $p\Vdash\dot{f}$ is a $\mathbb{P}$-name for a function from
$\Gamma$ into the ordinals. Then there is a function $g \in V$
and $q\leq p$ such that $q\Vdash \dot{f}(\alpha) \in g(\alpha)$ for
$\alpha < \Gamma$ and $\lvert g(\alpha)\lvert \leq \ka$ for all
$\alpha$. In particular, $\bbP$ preserves cofinalities and so
cardinalities.
\end{coroll}
\begin{proof}
Let $\mathcal{A}_\alpha$ be a maximal antichain of conditions below
$p$ deciding a value for $\dot{f}(\alpha)$. Use Lemma ~\ref{anti} to
find $q\leq p$ such
that $\mathcal{A}_\alpha \cap \mathbb{P} \downarrow q$ is a maximal
antichain in $\mathbb{P} \downarrow q$ for all $\alpha <
\Gamma$. Finally define the function $g \in V$ as follows: $g(\alpha)=
\{ \beta:
  \exists r \leq q $ such that $r \Vdash \dot{f}(\alpha)=\beta\}$.
\end{proof}

We now present the key lemmas that will allow us to construct the witness for
$\mathfrak{u}(\ka)= \ka^\ast$.

\begin{lem}\label{ultrafilter}
Let $\ka$ be a supercompact cardinal and $\ka^\ast$ be a cardinal satisfying
$\ka < \ka^\ast \leq \Gamma$, $\ka^\ast$ regular. Suppose that $p \in \mathbb{P}$ is such that $p \Vdash \dot{\mathcal{U}}$ is a normal ultrafilter on
$\ka$.\footnote{This is possible because $\ka$ is still supercompact in $V^{\mathbb{P}}$.} Then for some $\alpha<
\Gamma^+ $ there is an extension $q \leq p$ such that $q \Vdash$ $( \dot{\mathcal{U}}_\alpha = \dot{\mathcal{U}}
\cap V[G_\alpha])$. Moreover this can be done for a set of ordinals $S \subseteq \Gamma^+$ of
order type $\ka^\ast$ in such a way that $\forall \alpha \in S (\dot{\mathcal{U}} \cap V_\alpha \in
V[G_\alpha])$ and $\dot{\mathcal{U}} \cap V[G_{\sup S}] \in
V[G_{\sup S}]$. Here $\dot{\mathcal{U}}_\alpha$ is the canonical
name for the ultrafilter generically chosen at stage $\alpha$.
\end{lem}

\begin{proof} Let $\alpha_0=\sup\hbox{USupt}(p)$.
Then $\PP_{\alpha_0}\downarrow p$ is $\kappa^+$-cc and has a dense subset of size at most $\Gamma$. Thus there are just $\Gamma$-many
$\mathbb{P}_{\alpha_0} \downarrow p$-names for
subsets of $\ka$. Let $\bar{X}=(\dot{X}_i: i < \Gamma)$ be an enumeration of them.

We view each condition in $\mathbb{P}$ as having three main parts.
The first part corresponds to the choice of ultrafilters in even coordinates
--- the ``$\calU$''s of $r=(\calU,r_1)$ for iterand conditions $r$;
we call this the \emph{Ultrafilter Part}.  The next part corresponds to the
coordinates where we have in addition non-trivial Mathias conditions
(coordinates in $\hbox{SSupt}$),
we call it the \emph{Mathias part}.  Finally the odd coordinates, where the forcing chooses
conditions in an arbitrary $\ka$-centered, $\ka$-directed closed forcing (coordinates in $\hbox{RSupt}$),
we call the \emph{Directed Part}.

Extend $p_0=p$ to a condition $p_1$ deciding
whether $\dot{X}_0\in\dot{\calU}$, and let $p_1'$ be the condition extending
$p_0$ with the same ultrafilter part as $p_1$ and no other change from $p_0$.
Then extend
$p_1'$ again to a condition $p_2 $ which also makes a decision about
$\dot{X}_0$ but either
its Mathias or directed parts are incompatible with the ones corresponding to $p_1$; and correspondingly extend $p_1'$ on its ultrafilter part to $p_2'$.

Continue extending the ultrafilter part,
deciding whether or not $\dot{X}_0 \in \dot{U}$ with an
antichain of different Mathias and directed parts until a maximal antichain is reached.
This will happen
in less than $\ka^+$-many steps. If the resulting condition is called $q_1$ and has support
$\alpha_1 < \Gamma^+$ (without loss of generality it is an odd ordinal),
then the set of conditions in
$\mathbb{P}_{\alpha_1} \downarrow q_1$ which decide whether or not $\dot{X}_0$ belongs to
$\dot{\mathcal{U}}$ is predense in $\mathbb{P}_{\alpha_1}\downarrow q_1$.

Repeat this process $\Gamma$-many times for each element in $\bar{X}$ until reaching a condition
$q_2$ with the same property for all such names. Then do it for all $\mathbb{P}_{\alpha_1} \downarrow
q_2$ names for subsets of $\ka$ and so on.
Let $q$ be the condition obtained once this overall
process closes off with a fixed point.
It follows, that if $G$ is $\mathbb{P}$ generic containing $q$
then $\dot{\mathcal{U}}^G \cap V[G_\alpha]$ is determined by $G_\alpha$ and therefore it is a normal ultrafilter
$U_\alpha$ on $\ka$ in $V[G_\alpha]$. Now extend $q$ once more to length $\alpha+1$ by choosing
$\dot{\mathcal{U}}_\alpha$ to be the name for $\mathcal{U}_\alpha= \dot{U}^G \cap V[G_\alpha]$.

This argument gives us the desired property for a single $\alpha< \Gamma$. To have it
for all $\alpha \in S \cup \{\sup S\}$ we just have to iterate the process $\ka^\ast$-many times
(this is possible because $\ka^\ast< \Gamma$), and then
by cofinality considerations we see that
moreover $\dot{\mathcal{U}} \cap V[G_{\sup S}] \in
V[G_{\sup S}]$.
\end{proof}

\textbf{Remark:} Note that we can choose the domains of our conditions such that they have
size $\Gamma$.

Take $S$ to be a set with the properties of the lemma above; this set will be fixed
for the rest of the paper.

Now, using our Laver preparation $S_\ka$ and Laver function $h$ we choose a supercompactness
embedding $j^\ast: V \to M$ with critical point $\ka$
satisfying $j^\ast(\ka) \geq \lambda$ where $\lambda \geq \lvert S_\ka \ast \mathbb{P} \lvert$, $M ^\kappa
\subseteq M$ and $j^\ast(h)(\ka) = (\mathbb{P}, \lambda)$. Then $j^\ast(S_\kappa)
= S_\kappa \ast \dot{\mathbb{P}} \ast \dot{S}^\ast$ for an appropriate tail iteration $\dot{S}^\ast$ in $M$. Also if we denote $\mathbb{P}'= j^\ast
(\mathbb{P})$ applying $j^\ast$ to $S_\kappa \ast \dot{\mathbb{P}}$ we get $j^\ast(S_\kappa \ast \dot{\mathbb{P}})=
S_\kappa \ast \mathbb{P} \ast \dot{S}^\ast \ast (\mathbb{P}') ^M$.

Consider then
$j_0: V[G_{S_\ka}] \to M[G_{S_\ka}][G_{\mathbb{P}}][H]$ where $G_{S_\ka} \ast G_{\mathbb{P}} \ast H$
is generic for $j(S_\ka \ast \dot{\mathbb{P}})$.
We want to lift again to $j^\ast : V[G_{S_\ka}][G_\mathbb{P}] \to M[G_{S_\ka}][G_{\mathbb{P}}]
[H][G_{\mathbb{P}'}]$ where $\mathbb{P}'= j_0(\mathbb{P})$. We will do this by listing the maximal
antichains below some master condition in $\mathbb{P}'$ extending every condition of the form
$j_0(p)$ for $p \in G_\mathbb{P}$. The obvious master condition comes from choosing a lower
bound $p_0^\ast$ of $j_0[G_\mathbb{P}]$.\footnote{This exists because $j_0[G_\mathbb{P}]$ is directed
and the forcing is sufficiently directed-closed} 

This condition has support contained in $j[\Gamma^+]$ and for each $i < \Gamma^+$ odd
chooses the filter name $\dot{\mathcal{U}}_{j(i)}$ to be $j_0(\dot{\mathcal{U}}_i)$ as well as a
$j(\ka)$-Mathias name with first component $\check{x}_i$, the Mathias generic added by $G_\mathbb{P}$ at stage $i$ of the
iteration. However we will choose a stronger master condition $p^\ast$ with support contained in
$j[\Gamma^+]$ as follows.

($\ast$) If $i < \Gamma^+$ is an even ordinal and for each $A \in U_i$ there is a $G_{\mathbb{P}_i}$
-name $\dot{X}$ such that $A= X^{G_{\mathbb{P}_i}}$ and a condition $p \in G_{\mathbb{P}_i}$
such that $j_0(p)\Vdash \ka \in j_0(\dot{X})$, then $p^\ast(j(i))$ is obtained
from $p_0^\ast(j(i))$ by replacing the first component $x_i$ of its $j(\ka)$-Mathias name
by $x_i \cup \{ \ka\}$.

Otherwise $p^\ast(j(i)) =p_0^\ast(j(i))$.

\begin{lem} The condition $p^\ast$ is an extension on $p_0^\ast$.
If $G_{\mathbb{P}'}$ is chosen to contain $p^\ast$, $j^\ast$ is the resulting lifting of
$j_0$ and $\mathcal{U}$ is the resulting normal ultrafilter on $\ka$ derived from $j^\ast$,
then whenever $\mathcal{U}_i$ is contained in $\mathcal{U}$, we have that $x_i \in \mathcal{U}$
\end{lem}

\begin{proof}
To show the first claim, it is enough to show that for all $i < \Gamma^+$ the condition
$p_i^\ast$ defined as $p^\ast$ but replacing $x_{j(l)}$ by $x_{j(l)} \cup \{ \ka \}$
for $l < i$ satisfying ($\ast$) extends $p_0^\ast$. We do this by induction on $i$.
The base and limit cases are immediate. For the successor one, suppose we have the result for $i$
and we want to prove it for $i+1$. Let $G_{\mathbb{P}^\ast_j(i)}$ be any generic containing
$p_i^\ast\restriction j(i)$ and extend it to a generic $G_{\mathbb{P}^\ast}$ containing $p_i^\ast$.
Hence, using the induction hypothesis $G_{\mathbb{P}^\ast}$ also contains $p_0^\ast$ and
therefore gives us a lifting $j^\ast$ of $j_0$.

Now, any $p \in G_{\mathbb{P}}$ can be extended (inside $G_{\mathbb{P}}$) so that the Mathias
condition it specifies at stage $i$ is of the form $(s,A) \in \mathbb{M}^\ka_{\mathcal{U}_i}$ where
$s \subseteq x_i$ and $A \in \mathcal{U}_i$. Then using ($\ast$) we infer $A= X^{G_{\mathbb{P}_i}}$
where $j_0(q) \Vdash \ka \in j_0(\dot{X})$ for some $q\in G_{\mathbb{P}_i}$.

But then, since $p_0^\ast \in G_{\mathbb{P}^\ast}$, $j_0(q)$ is an element of $G_{\mathbb{P}^\ast_j(i)}$
and therefore

\centerline{$\ka \in j_0(\dot{X})^{G_{\mathbb{P}^\ast_{j(i)}}}= j^\ast(A)$.}

It follows that the $j(\ka)$-Mathias condition specified by $p_{i+1}^\ast(j(i))^{G_{\mathbb{P}^\ast_{j(i)}}}$ with first component $x_i \cup \{ \ka\}$ does extend

\centerline{$(x_i, j^\ast(A))= (x_i, j_0(\dot{X})^{G_{\mathbb{P}^\ast_{j(i)}}})
\leq (s, j_0(\dot{X})^{G_{\mathbb{P}^\ast_{j(i)}}})$.}

This means that $p_i^\ast \restriction j(i) \Vdash p_{i+1}^\ast (j(i)) \leq (s, j_0(\dot{X}))= j_0(p)(j(i))$ and thus $p_{i+1}^\ast$ extends $j_0^\ast(p)$ for each $p \in G_{\mathbb{P}}$ and then also extends $p_0^\ast$.

To see the second claim, note that  if $\mathcal{U}_i \subseteq \mathcal{U}$, then $\ka \in j^\ast(A)$ for all $A \in \mathcal{U}_i$
which implies that ($\ast$) is satisfied at $i$. Then $\ka \in j^\ast(x_i)$ and so $x_i \in \mathcal{U}$.
\end{proof}

\begin{thm}\label{thm_uc}
Suppose $\kappa$ is a supercompact cardinal and $\kappa^\ast$ is a regular cardinal with
$\kappa <\kappa^{\ast} \leq  \Gamma$, $\Gamma^\kappa= \Gamma$.
There is a forcing notion $\mathbb{P}^\ast$ preserving cofinalities such that
$V^{\mathbb{P}^\ast} \models \mathfrak{u}(\kappa)= \kappa^\ast \wedge 2^\ka = \Gamma$.
\end{thm}
\begin{proof}
We will not work with the whole generic extension given by $\mathbb{P}$. In fact we will chop the
iteration in the step $\alpha= \sup(S)$ (as in the Lemma~\ref{ultrafilter}) this is an
ordinal of cofinality $\kappa^\ast$. Define $\mathbb{P}^\ast= \mathbb{P}_{\alpha}$.

Take $G$ to be a $\mathbb{P}^\ast$-generic filter, the fact that $2^\ka= \Gamma$ is a consequence of
the fact that, the domains of the conditions obtained in Lemma~\ref{ultrafilter} can be chosen
in such a way that they all have size $\Gamma$.

To prove $\mathfrak{u}(\kappa)= \kappa^\ast$ we consider the ultrafilter $\mathcal{U}^\ast$ on $\kappa$ given by
the restriction of $\mathcal{U}$ (Lemma~\ref{ultrafilter}). Then by the same lemma note that for all
$i \in S$ the restriction of $\mathcal{U}$ to the model $V[G_i]$ belongs to $V[G_{i+1}]$ and moreover,
this is the ultrafilter $U_i^G$ chosen generically at stage $i$.

Furthermore by our choice of Master Conditions the $\ka$-Mathias generics $\dot{x}_i$ belong to
$\mathcal{U}$. Then $\mathcal{U}^\ast$ is generated by $\dot{x}_i$ for $i \in S$.

The other inequality $\mathfrak{u}(\kappa)\geq \kappa^\ast$ is a consequence of $\mathfrak{b}(\kappa)\geq \kappa^\ast$ and Proposition~\ref{mm}.
\end{proof}

\begin{prop}\label{mm}
$\mathfrak{b}(\kappa) \leq \mathfrak{r}(\kappa)$ and $\mathfrak{r}(\kappa) \leq \mathfrak{u}(\kappa)$.
\end{prop}

\begin{proof}
The first is the consequence of the following property that can be directly generalized from the countable case: there are functions $\Phi: [\ka]^\ka \to \ka^{\uparrow \ka}$ and
\mbox{$\Psi: \ka^{\uparrow \ka} \to [\ka]^\ka$}
such that whenever $\Phi(A) \leq^\ast f$ then $\Psi(f)$ splits $A$.

For the second one, it is just necessary to notice that if $\mathcal{B}$ is a base for a uniform ultrafilter on $\ka$, then $\mathcal{B}$ cannot be split by a single set $X$. Otherwise neither $X$ nor $\ka \setminus X$ will belong to the ultrafilter.
\end{proof}

\section{The generalized cardinal characteristics}

In the following subsections 4.1 - 4.6 we systemize those properties of the generalized cardinal characteristics which will be of importance for
our main consistency result.

\subsection{Unbounded and Dominating Families in $^\kappa\kappa$}

\begin{defn}
For two functions $f,g \in \ka^\ka$, we say $f \leq^\ast g$ if and only if there exists $\alpha< \kappa$ such that for all $\beta> \alpha$, $f(\beta) \leq g(\beta)$. A family $\mathfrak{F}$ functions from $\kappa$ to $\kappa$ is said to be
\begin{itemize}
 \item dominating, if for all $g \in \kappa^\kappa$, there exists an $f \in \mathfrak{F}$ such that $g \leq^{\ast} f$.
 \item unbounded, if for all $g \in \kappa^\kappa$, there exists an $f \in \mathfrak{F}$ such that $f \nleq^{\ast} g$.
\end{itemize}
\end{defn}

\begin{defn}
The unbounding and dominating numbers, $\mathfrak{b}(\kappa)$ and $\mathfrak{d}(\kappa)$ respectively are defined as follows:

\begin{itemize}
 \item $\mathfrak{b}(\kappa)= \min\{ \lvert \mathfrak{F} \lvert: \mathfrak{F}$ is an unbounded
 family of functions from $\kappa$ to $\kappa \}$.
 \item $\mathfrak{d}(\kappa)= \min\{ \lvert \mathfrak{F} \lvert: \mathfrak{F}$ is a dominating
 family of functions from $\kappa$ to $\kappa \}$.
\end{itemize}
\end{defn}

\begin{defn}[Generalized Laver forcing]
Let $\mathcal{U}$ be a $\kappa$-complete non-principal ultrafilter on $\kappa$.
\begin{itemize}
\item A $\mathcal{U}$-Laver tree is a $\kappa$-closed tree $T \subseteq \kappa^{<\kappa}$ of increasing sequences with the property that $\forall s \in T(\lvert s\lvert \geq \lvert\stem(T)\lvert\to \Succ_T(s) \in \mathcal{U})\}$.
\item The generalized Laver Forcing $\mathbb{L}^\ka_{\mathcal{U}}$ consists of all $\mathcal{U}$-Laver trees with order given by inclusion.
\end{itemize}
\end{defn}

\begin{prop}
Generalized Laver forcing $\mathbb{L}^\ka_{\mathcal{U}}$ generically adds a dominating function from $\kappa$ to $\kappa$.
\end{prop}
\begin{proof}
Let $G$ be a $\mathbb{L}^\kappa_{\mathcal{U}}$-generic filter. The Laver generic function in
$\kappa^\kappa$, $l_G$, is defined as follows: $l_G = \cap \{ [T] : T \in G \}$ where $[T]$ is
the set of branches in $T$.

To show that $l_G$ is a dominating function it is enough to notice that, for all $f \in \kappa^\kappa$ and all $T\in\mathbb{L}^\kappa_{\mathcal{U}}$,
the set $T_f = \{ s \in T: \forall \alpha
(( \lvert \stem(T)\lvert \leq \alpha < \lvert s\lvert) \to s(\alpha) > f(\alpha))\}$ is also a condition in $\mathbb{L}^\kappa_{\mathcal{U}}$ and $T_f \leq T$.
By genericity we conclude that $ V[G] \models \forall f \in V \cap \kappa^\kappa (f \leq^\ast l_G)$.
\end{proof}

\begin{lem}
If $\mathcal{U}$ is a normal ultrafilter on $\kappa$, then $\mathbb{M}^\kappa_{\mathcal{U}}$ and $\mathbb{L}^\kappa_{\mathcal{U}}$ are forcing equivalent.
\end{lem}
\begin{proof}
The main point that we will use in this proof is that, when $\mathcal{U}$ is normal we have the following
``Ramsey''-like
property: For all $f: [\kappa]^{<\omega}
 \to \gamma$ where $\gamma < \kappa$, there is a set in $\mathcal{U}$ homogeneous for $f$.

Also it is worth to remember that in the countable case if $\mathcal{U}$ is a Ramsey Ultrafilter $\mathbb{M}(\mathcal{U}) \simeq \mathbb{L}(\mathcal{U})$.
Thus, we want to define a dense embedding $\phi: \mathbb{M}^\kappa_{\mathcal{U}} \to \mathbb{L}^\kappa_{\mathcal{U}}$. Take $(s,A)$ a condition in
$\mathbb{M}^\kappa_{\mathcal{U}}$ and define the tree $T= T_{(s,A)}$ as follows:

\begin{itemize}
  \item $\sigma= \stem(T)$ will be the increasing enumeration of $s$.
  \item If we already have constructed $\tau \in T_\alpha$, with $\tau \supseteq \sigma$, then
  $\tau^\smallfrown \langle \alpha \rangle \in T_{\alpha+1}$ if and only if
  $\alpha \in A$ and $\alpha \geq \sup\{ \tau(\beta): \beta < \alpha\}$.
  \item In the limit steps just ensure that $\tau \in T_\alpha$ if and only if $\tau \restriction \beta \in T_\beta$.
\end{itemize}

Note that $T$ is a condition in $\mathbb{L}^\kappa_{\mathcal{U}}$. For the limit steps note that if $\tau \in T_\alpha$ for $\alpha$ limit,
 then the set $\Succ_T(\tau) \supseteq \bigcap_{\beta< \alpha} \Succ_T(\tau \restriction \beta)$.

Now, consider the map $\phi: (s,A) \to T_{(s,A)}$. Since this map preserves $\leq$, it is enough to prove that the trees of the form
$T_{(s,A)}$ are dense in $\mathbb{L}^\kappa_{\mathcal{U}}$. For that, take an arbitrary $T \in \mathbb{L}^\kappa_{\mathcal{U}}$ and define:

\begin{equation*}
 f(\{\alpha,\beta\})=\begin{cases}
                      1 &

              \begin{aligned}[r]        \text{if }  \forall s \in T \text{ with } \alpha \geq \sup\{ s(\gamma): \gamma< \lvert
             s\lvert \} \\ ( \alpha \leq \beta \rightarrow \beta \in \Succ_T(s)) \\
             \end{aligned}
             \\ 0 &  \text{otherwise} \\
             \end{cases}
\end{equation*}

Using the Ramsey-like property we can find a set $B \in \mathcal{U}$ homogeneous for $f$. The color of $B$
cannot be $0$ because $T$ is a Laver tree.
Now, knowing that $f''[B]^2 =\{1\}$, we can define $s= \ran(\stem(T))$ and
$A= B \cap \Succ_T(\stem(T))$ and conclude that $T_{(s,A)} \leq T$ as we wanted.
\end{proof}

\begin{coroll}
If $\mathcal{U}$ is a normal ultrafilter on $\kappa$ then $\mathbb{M}^\kappa_{\mathcal{U}}$ always adds dominating functions.
\end{coroll}

\subsection{$\kappa$-maximal almost disjoint families}

\begin{defn}
Two sets $A$ and $B \in \mathcal{P}(\kappa)$ are called $\kappa$-almost disjoint if $A \cap B$ has size $<\kappa$. We say that a family of sets
$\mathcal{A} \subseteq \mathcal{P}(\kappa)$ is $\kappa$-almost disjoint if it has size at least $\kappa$ and all its elements are pairwise $\kappa$-almost disjoint. A family $\mathcal{A} \subseteq [\kappa]^\kappa$ is called a $\kappa$-maximal almost disjoint (abbreviated $\kappa$-mad) if it is  $\kappa$-almost disjoint and is not properly included in another such family.
\end{defn}

\begin{defn} $\mathfrak{a}(\kappa)= \min\{\lvert \mathfrak{A} \lvert: \mathfrak{A}$ is a $\kappa$-mad family$ \}$
\end{defn}

\begin{prop}
 $\mathfrak{b}(\kappa) \leq \mathfrak{a}(\kappa)$
\end{prop}
\begin{proof}
Suppose $ \mathfrak{a}(\kappa)= \lambda$, let  $\mathfrak{A}= \{ A_{\alpha}: \alpha
< \lambda \}$ be a
$\kappa$-almost disjoint family where $\lambda < \mathfrak{b}(\kappa)$. For each $\alpha < \kappa$,
let $\tilde{A}_{\alpha}= A_{\alpha} \setminus \bigcup_{\delta < \alpha }
(A_{\alpha} \cap A_{\delta})$. Since $\mathfrak{A}$ is $\kappa$-ad, we have $|\tilde{A}_{\alpha}|=
\kappa$, also $\tilde{A}_{\alpha} \cap \tilde{A}_{\beta}= \emptyset$ for all $\alpha, \beta < \kappa$.
Thus, $\tilde{A}_{\alpha}=^{\ast} A_{\alpha}$. (Here $\ast$ means modulo a set of size $<\kappa$).

Whenever $g \in \kappa^{\kappa}$, define $e_{g}^{\alpha}= \next(\tilde{A}_{\alpha}, g(\alpha))$,
the least ordinal in $\tilde{A}_{\alpha}$ greater than $g(\alpha)$. Let $E_{g}= \{e_{g}^{\alpha}:
\alpha < \kappa \}$. Then $E_{g}$ contains one element of each $\tilde{A}_{\alpha}$, so it is unbounded in $\kappa$.
Also $\lvert E_{g} \cap A_{\alpha}\lvert < \kappa$, for all $\alpha < \kappa$.

Now when $\kappa \leq \alpha < \lambda$. Each $A_{\alpha} \cap A_{\gamma}$, has size less than
$\kappa$, so we can fix $f_{\alpha}$ such that for all $\gamma < \kappa$ all elements of $A_{\alpha} \cap A_{\gamma}$ are less than $f_{\alpha}(\gamma)$. Where $f_{\alpha}(\gamma)= \sup (A_{\alpha} \cap A_{\gamma}) +1 $.

Now consider $\{f_{\alpha}: \alpha < \lambda\}$, which is a family of $\lambda < \mathfrak{b}(\kappa)$ functions, therefore there exists $g \in \kappa^{\kappa}$ with the property $f_{\alpha} <^{\ast} g$, for all $\alpha$.

As consequence we have that $E_{g} \cap A_{\alpha}$ has size less than $\kappa$, for all $\alpha$ because if $e_{g}^{\gamma} \in E_{g} \cap A_{\alpha}$ then $e_{g}^{\gamma} \in \tilde{A}_{\alpha}$ and $e_{g}^{\gamma}> g(\alpha)$, so $f_{\alpha}(\gamma)> e_{g}^{\gamma} > g(\gamma) $ which is only possible for a set of less than $\kappa$ values.

Therefore, $\mathfrak{A}$ is not maximal. Then $\mathfrak{b}(\kappa) \leq \lambda$.
\end{proof}

\begin{defn}\label{mad_poset}
Let $\mathcal{A}= \{A_i\}_ {i< \delta}$ be a $\kappa$-almost disjoint family. Let $\bar{\QQ}(\mathcal{A},\kappa)$ be the poset of all
pairs $(s,F)$ where $s \in 2^{<\ka}$ and $F \in [\mathcal{A}]^{< \ka}$,
with extension relation stating that $(t,H) \leq (s,F)$ if and only if
$t \supseteq s$, $H \supseteq F$ and for all $i \in \dom(t) \setminus \dom(s)$ with $t(i)=1$ we have $i \notin \bigcup\{A: A \in F\}$
\end{defn}

Note that the poset $\bar{\QQ}(\mathcal{A},\kappa)$ is $\ka$-centered and $\ka$-directed closed. If $G$ is $\bar{\QQ}(\mathcal{A},\kappa)$-generic then $\chi_G= \bigcup \{t: \exists F (t,F) \in G \}$ is the characteristic function of an unbounded subset $x_G$ of $\ka$ such that $\forall A \in \mathcal{A} (\lvert A \cap x_G \lvert) < \ka$.

\begin{prop}\label{mad}
If $Y \in [\ka]^\ka \setminus \mathcal{I}_\mathcal{A}$, where $\mathcal{I}_\mathcal{A}$ is the $\ka$-complete ideal generated by the $\ka$-ad-family $\mathcal{A}$, then $\Vdash_{\Theta(\mathcal{A},\kappa)}\lvert Y \cap \dot{x}_G \lvert = \ka$.
\end{prop}
\begin{proof} Let $(s, F) \in \bar{\QQ}(\mathcal{A},\kappa)$ and $\alpha < \ka$ be arbitrary. It is sufficient to show that there are $(t,H) \leq (s,F)$ and $\beta > \alpha$ such that $(t,H) \Vdash \beta \in \check{Y} \cap \dot{x}_G$. Since $\ka \setminus \bigcup F$ is unbounded and $Y \notin \mathcal{I}_\mathcal{A}$, we have that $\lvert Y \setminus \bigcup F \lvert =\ka$. Take any $\beta> \alpha$ in $Y \setminus \bigcup F$ and  define $t'= t \cup \{ (\beta, 1)\} \cup \{ (\gamma, 0):\sup(\dom(t)) <\gamma< \beta\}$. Then $(t', H)$ is as desired.
\end{proof}

\subsection{The Generalized Splitting, Reaping and Independence Numbers}

\begin{defn} For $A$ and $B \in \wp(\kappa)$, say $A \subseteq^\ast B$ ($A$ is almost contained in $B$) if $A \setminus B$ has size $<\kappa$.
We also say that $A$ splits $B$ if both $A \cap B$ and $B \setminus A$ have size $\kappa$. A family $\mathcal{A}$ is called
a splitting family if every unbounded (with supremum $\kappa$) subset of $\kappa$ is split by a member of $\mathcal{A}$.
Finally $\mathcal{A}$ is unsplit if no single set splits all members of $\mathcal{A}$.
\begin{itemize}
\item $\mathfrak{s}(\kappa)= \min\{ \lvert \mathcal{A} \lvert: \mathcal{A}$ is a splitting  family of subsets of $\kappa \}$.
\item $\mathfrak{r}(\kappa)= \min\{ \lvert \mathcal{A} \lvert: \mathcal{A}$ is an unsplit family of subsets of $\kappa \}$.
\end{itemize}
\end{defn}

\begin{defn}
A family $\mathcal{I}=\{ I_\delta: \delta < \mu\}$ of subsets of $\ka$ is called $\ka$-independent
if for all disjoint $I_0, I_1 \subseteq \mathcal{I}$, both of size $< \kappa$,  $\bigcap_{\delta \in I_0}
I_\delta$ $\cap$ $\bigcap_{\delta \in I_0} (I_\delta)^c$ is
unbounded in $\ka$. The generalized independence number $\mathfrak{i}(\kappa)$ is defined as the minimal size of a
$\kappa$-independent family.
\end{defn}

\begin{prop}\label{indep}
If $\mathfrak{d}(\kappa)$ is such that for every $\gamma<\mathfrak{d}(\kappa)$ we have $\gamma^{<\kappa}<\mathfrak{d}(\kappa)$, then $\mathfrak{d}(\kappa) \leq \mathfrak{i}(\kappa)$
 \end{prop}

The proof will be essentially a modification of the one for the countable case (Theorem 5.3 in \cite{AB:CCC}). To obtain the above proposition, we will need the following lemma.

\begin{lem}
  Suppose $\mathcal{C}=(C_\alpha: \alpha< \kappa)$ is a $\subseteq^\ast$-decreasing sequence of unbounded subsets of $\ka$
  and $\mathcal{A}$ is a family of less than $\mathfrak{d}(\ka)$ many subsets of $\ka$ such that each set in
  $\mathcal{A}$ intersects every $C_\alpha$ in a set of size $\ka$.
  Then $\mathcal{C}$ has a pseudointersection $B$ that also has unbounded intersection with each member
  of $\mathcal{A}$.
 \end{lem}

 \begin{proof}
 Without loss of generality assume that the sequence $\mathcal{C}$ is $\subseteq$-decreasing.
 For any $h \in \ka^\ka$ define $B_h = \bigcup_{\alpha< \ka}(C_\alpha \cap h(\alpha))$, clearly $B_h$
 is a pseudointersection of $\mathcal{C}$. Thus, we must find $h \in \ka^\ka$ such that $\lvert B_h
 \cap A\lvert = \ka$ for each $A \in \mathcal{A}$.

 For each $A \in \mathcal{A}$ define the function $f_A \in \ka^\ka$ as follows: $f_A(\beta) =$ the $\beta$-th
 element of $A \cap C_\beta$. The set $\{f_A: A \in \mathcal{A}\}$ has cardinality $< \mathfrak{d}(\ka)$, then we can find $h \in \ka^\ka$ such
 that for all $A \in \mathcal{A}$, $h \nleq^\ast f_A$ (i.e. $X_A =\{\delta< \kappa: f_A(\delta)<
 h(\delta)\}$ is unbounded).

 Then $B_h$ will be the pseudointersection we need. Note that $B_h \cap A= \bigcup_{\alpha< \ka}
 (C_\alpha \cap A) \cap h(\alpha) \supseteq \bigcup_{\alpha \in X_A}
 (C_\alpha \cap A) \cap f_A(\alpha)$ which is unbounded.
 \end{proof}

\begin{proof}[Proof of Proposition ~\ref{indep}]
Suppose that $\mathcal{I}$ is an independent family of cardinality $< \mathfrak{d}(\ka)$, we will show it is not maximal.
For this purpose choose $\mathcal{D}=(D_\alpha:\alpha< \ka)\subseteq \mathcal{I}$ and let $\mathcal{I}'=\mathcal{I} \setminus \mathcal{D}$.

For each $f: \ka \to 2$ consider the set $C_\alpha =\bigcap_{\beta< \alpha} D_\beta^{f(\beta)}$ where $D^0 = D$ and $D^1= D^c$, also define $\mathcal{A}= \{ \bigcap I_0 \setminus \bigcup I_1: I_0 $ and $I_1$ are disjoint subfamilies of $\mathcal{I}$ of size $ <\ka\}$. Since $|\mathcal{I}|^{< \kappa} < \mathfrak{d}(\kappa)$, the family $\mathcal{A}$ has size $< \mathfrak{d}(\kappa)$.

Then, using the lemma before there exists a pseudointersection $B_f$ of the family $(C_\alpha: \alpha <
\ka)$ that intersects in an unbounded set all members of $\mathcal{A}$. Then if $f \neq g$ we have
$\lvert B_f \cap B_g \lvert < \ka$ (Moreover, we can suppose they are disjoint).

Now, fix two disjoint dense subsets $X$ and $X'$ of $2^\ka$. Take $Y= \bigcup_{f \in X} B_f$ and $Y'= \bigcup_{f \in X'} B_f$, note that $Y \cap Y' = \emptyset$. Then it is enough to show that both $Y$ and $Y'$ have intersection of size $\ka$ with each member of $\mathcal{A}$. We write the argument for $Y$ (for $Y'$ i is analogous).

Take $J_0, J_1 \subseteq \mathcal{I}$ both of size $< \kappa$, call $J_0', J_1'$ their intersections with $\mathcal{I}'$. There exists $\alpha < \kappa$ such that if $D_\beta$ belongs to $J_0$ or $J_1$, then $\beta < \alpha$ and using the density of the sets $X$ fix $f \in X$ such that, if $D_\beta \in J_0 \cup J_1$, then $f(\beta) = 0$ or $1$ respectively. Hence:

\begin{multline}
\bigcap J_0 \setminus \bigcup J_1 = \bigcap J_0' \setminus \bigcup J_1' \cap \bigcap_{\{\beta:
D_\beta \in J_0 \cup J_1\}} D_\beta^{f(\beta)} \\
\supseteq \bigcap J_0' \setminus \bigcup J_1' \cap \bigcap_{\beta< \alpha}
D_\beta^{f(\beta)}\\
{^\ast}\supseteq \bigcap J_0' \setminus \bigcup J_1' \cap B_f \text{ which is unbounded.}
\end{multline}
\end{proof}

\begin{lem}\label{i}
Let $\mathcal{I}$ be an independent family of size $\ka$. Then there is a $\ka$-centered
forcing notion $\hat{\QQ}(\mathcal{I},\kappa)$ that adds a set $Y \in [\ka]^\ka$ such that:
\begin{enumerate}
  \item in $V^{\hat{\QQ}(\mathcal{I},\kappa)}$, $\mathcal{I} \cup \{Y\}$ is independent;
  \item $\forall Z \in V \cap [\ka]^\ka$ such that $Z \notin \mathcal{I}$,
	 $ V^{\hat{\QQ}(\mathcal{I},\kappa)} \models \mathcal{I} \cup \{Z,Y\}$ is not independent.
\end{enumerate}
 \end{lem}
\begin{proof}
Let $\mathcal{B}_\mathcal{I}$ be the Boolean algebra generated by $\mathcal{I}$. Note that $\mathcal{B}_{\mathcal{I}}$ is $\kappa$-complete. Since $\mathcal{I}$ is not maximal, there is $X_0\subseteq\kappa$ such that for all $B\in\mathcal{B}_{\mathcal{I}}$ both $B\cap X$ and $B\cap X^c$ are of size $\kappa$. Thus in particular, $\mathcal{I}\cup\{X_0\}$ is independent. Recursively construct an increasing chain $\{I_\alpha\}_{\alpha<\delta}$ of independent families and a family $\mathcal{X}=\{X_\alpha\}_{\alpha<\delta}\subseteq[\kappa]^\kappa$ such that
\begin{enumerate}
\item $\mathcal{I}_{\alpha+1}=\mathcal{I}_\alpha\cup\{X_\alpha\}$; if $\alpha$ is a limit then
$\mathcal{I}_\alpha=\bigcup_{\beta<\alpha}\mathcal{I}_\alpha$;
\item $\forall\alpha<\delta\forall B\in\mathcal{B}_{\mathcal{I}_\alpha}$ we have $|X_\alpha\cap B|=|X_\alpha^c\cap B|=\kappa$.
\end{enumerate}
Then in particular $\mathcal{X}$ forms a $\kappa$-complete filter base. Extend $\mathcal{X}$ to a $\kappa$-filter $\mathcal{G}$, which is maximal with respect to the following property:
$$\forall X\in\mathcal{G}\forall B\in\mathcal{B}(\mathcal{I})(|B\cap X|=|B\cap X^c|=\kappa).$$
Thus in particular for all $X\notin\mathcal{G}$ either there is $Z\in\mathcal{G}$ such that $X\cap Z$ is of size $<\kappa$, or there is $B\in\mathcal{B}_{\mathcal{I}}$ such that either $|Y\cap B|<\kappa$, or $|Y^c\cap B|<\kappa$.

Let $\hat{\QQ}(\mathcal{I},\kappa):=\mathbb{M}^\kappa_{\mathcal{G}}$ and let $G$ be $\QQ$-generic. Then $Y:=x_G=\bigcup\{s:\exists F(s,F)\in G\}$ is as desired. Indeed.
To see $(a)$ note that it is enough to show that for all $B \in \mathcal{B}(\mathcal{I})$ both $x_G \cap B$ and $x_G^c \cap B$ are forced to be unbounded in $\ka$. That $\Vdash |\dot{x}_G^c\cap \check{B}|=\kappa$ follows from the fact that given $B \in \mathcal{B}(\mathcal{I})$, $\Vdash \check{X} \cap \check{B}\subseteq^\ast \dot{x}_G^c \cap \check{B}$ for arbitrarily $X \in \mathcal{G}$. To see that $\Vdash |x_G\cap B|=\kappa$, proceed by contradiction.
That is suppose there is $B\in\mathcal{B}_\mathcal{I}$, $(s,A)\in\hat{\QQ}(\mathcal{I},\kappa)$ and $\alpha<\kappa$ such that $(s,A)\Vdash\dot{x}_G\cap\check{B}\subseteq\check{\alpha}$.
Since $B\cap A$ is unbounded in $\kappa$, we can choose $\beta\in B\cap A$ such that $\beta>\max\{\sup(s),\alpha\}$. Then  $(s \cup \{\beta\}, A \setminus (\beta +1)) \leq (s,A)$ and $(s \cup \{\beta\}, A \setminus (\beta +1)) \Vdash \beta\in \dot{x}_G \cap \check{B}$ which is a contradiction.

To see part $(b)$, take $Z \in (V \cap [\ka]^\ka) \setminus \mathcal{I}$.  If $Z \in \mathcal{G}$, then $\Vdash \dot{x}_G \subseteq^\ast Z$ and so $\Vdash |\dot{x}_G^c \cap Z|<\kappa$. If $Z \notin \mathcal{G}$, then either there is $X\in\mathcal{G}$ such that $|X\cap Z|<\kappa$ and so $\Vdash |\dot{x}_G\cap\check{Z}|<\kappa$, or there is $B\in\mathcal{B}_\mathcal{I}$ such that $|X\cap B|<\kappa$ or $|X^c\cap B|<\kappa$. Therefore
$\Vdash(\mathcal{I}\cup\{Z,\dot{x}_G\}\;\hbox{is not independent})$.
\end{proof}

\subsection{The generalized pseudointersection and tower numbers}

\begin{defn}
Let $\mathcal{F}$ be a family of subsets of $\kappa$, we say that $\mathcal{F}$ has the strong intersection property
(SIP) if any subfamily $\mathcal{F}' \subseteq \mathcal{F}$ of size $< \ka$ has intersection of size $\ka$, we also say that
$A \subseteq \ka $ is a pseudointersection of $\mathcal{F}$ is $A \subseteq^\ast F$, for all $F \in \mathcal{F}$. A tower $\mathcal{T}$ is a well-ordered family of subsets of $\kappa$ that has no pseudointersection of size $\ka$.
\begin{itemize}
\item  The generalized pseudointersection number $\mathfrak{p}(\kappa)$ is defined as  the minimal size of a family $\mathcal{F}$ which has the SIP but no pseudointersection of size $\kappa$.
\item The generalized tower number $\mathfrak{t}(\kappa)$ is defined as the minimal size of a tower $\mathcal{T}$  of subsets of $\kappa$.
\end{itemize}
\end{defn}

\begin{lem}\label{Lema36}
 $\kappa^+ \leq \mathfrak{p}(\kappa) \leq \mathfrak{t}(\kappa) \leq \mathfrak{b}(\ka)$
\end{lem}

\begin{proof}
 First we prove $\ka^+ \leq \mathfrak{p}(\ka)$: Take a family of subsets of $\ka$, $\mathcal{B}= (B_\alpha : \alpha< \ka)$ with the SIP. Then we can
 construct a new family $\mathcal{B}'=(B'_\alpha : \alpha< \ka)$ such that $B'_{\alpha+1} \subseteq B'_\alpha$ and $B'_\alpha \subseteq B_\alpha$ for all $\alpha< \kappa$. Simply define $B'_0= B_0$, $B'_{\alpha+1}= B_{\alpha+1} \cap B'_\alpha$ and for limit $\gamma$, $B'_\gamma= \bigcap_{\alpha< \gamma} B'_\alpha$. Note that this construction is possible thanks to the SIP.

 Then, without loss of generality we can find $\kappa$-many indexes $\beta$
 where it is possible to choose $a_\beta \in B'_{\alpha} \setminus B'_{\alpha+1}$. Hence the set $X = \{ a_\beta : \beta < \ka\}$ is a pseudointersection of the family $\mathcal{B}'$ and so of $\mathcal{B}$.

 $\mathfrak{p}(\ka) \leq \mathfrak{t}(\ka)$ is immediate from the definition and, $\mathfrak{t}(\ka) \leq \mathfrak{b}(\ka)$ was proven in Claim 1.8, \cite{SG:pity}.
\end{proof}

\subsection{The generalized distributivity number}

\begin{defn} The Generalized Distributivity Number $\mathfrak{h}(\kappa)$ is defined as the minimal $\lambda$ for which
$\mathcal{P}(\kappa) \big/ <\kappa$ is not $\lambda$-distributive. A poset $\mathbb{P}$ is $\lambda$-distributive if given a collection $\calD$ of $\lambda$-many dense open sets with the property that any subfamily of size $<\lambda$ has open dense intersection, has open dense intersection.
\end{defn}

\begin{prop}\label{prop31}
$\mathfrak{t}(\kappa) \leq \mathfrak{h}(\kappa) \leq \mathfrak{s}(\kappa)$
\end{prop}
\begin{proof} $\hbox{}$

\noindent
$\bullet$
$\mathfrak{t}(\kappa) \leq \mathfrak{h}(\kappa)$: Let $\delta < \frt(\ka)$ and $D_\alpha$ for
$\alpha < \delta$ a collection of open dense sets in $\mathcal{P}(\kappa) \big/ <\kappa$ such that every subfamily of size $<\delta$ has open dense intersection. Fix $A \in [\ka]^\ka$ and recursively define $A_\alpha$, $\alpha \leq \delta$
with $A_0= A$ and $A_\beta \subseteq^\ast A_\alpha$ for all $\beta > \alpha$ and $A_{\alpha+1} \in D_\alpha$. (In the
limit steps this is possible because $\delta < \frt(\ka)$). Finally $A_\delta \in \bigcap_{\alpha < \delta} D_\alpha$.

\bigskip
\noindent
$\bullet$
$\mathfrak{h}(\kappa) \leq \mathfrak{s}(\kappa)$: Let $\mathcal{S}$ be an splitting family of
subsets of $\ka$. For each $S \in \mathcal{S}$, the set $D_S = \{ X \in [\ka]^\ka: X$ is not split by $S\}$ is dense open. Because $\mathcal{S}$ is a splitting family we obtain $\bigcap_{S \in \mathcal{S}} D_S = \emptyset$.
\end{proof}

\subsection{Cardinals from Cich\'{o}n's diagram at $\ka$}

When $\ka$ is uncountable and satisfies $\ka^{<\ka}=\ka$, it is possible to endow $2^\kappa$ with the topology generated by the sets of the form $[s]= \{ f \in 2^\ka: f \supseteq s\}$, for $s \in 2^{<\ka}$. Then it is possible to define nowhere dense sets and meager sets as $\kappa$-unions of nowhere dense sets. Hence, we can
consider the Meager Ideal $\mathcal{M}_\ka$ and study the cardinal invariants associated to this ideal.
Specifically we are interested in the cardinals in Cich\'{o}n's Diagram.

\begin{itemize}
 \item $\add(\mathcal{M}_\kappa) = \min\{ \lvert \mathcal{J} \lvert: \mathcal{J} \subseteq \mathcal{M}_\ka$ and $ \cup \mathcal{J} \notin \mathcal{M}_\ka \}$
 \item $\cov(\mathcal{M}_\kappa)= \min \{\lvert \mathcal{J} \lvert: \mathcal{J} \subseteq \mathcal{M}_\ka$ and $ \cup \mathcal{J} = 2^\kappa \} $
 \item $\cof(\mathcal{M}_\kappa)= \min \{\lvert \mathcal{J} \lvert: \mathcal{J} \subseteq \mathcal{M}_\ka\;\hbox{and}\;\forall M \in \mathcal{M}_\ka \exists J \in \mathcal{J}\;\hbox{s.t.}\;M \subseteq J\}$
 \item $\non(\mathcal{M}_\kappa)= \min \{\lvert X\lvert: X \subset 2^\kappa$ and $ X \notin \mathcal{M}_\ka \}$
\end{itemize}

If in addition $\ka$ is strongly inaccessible we have a similar diagram as in the countable case (For specific details about these properties see \cite{BBFM:Icho}):

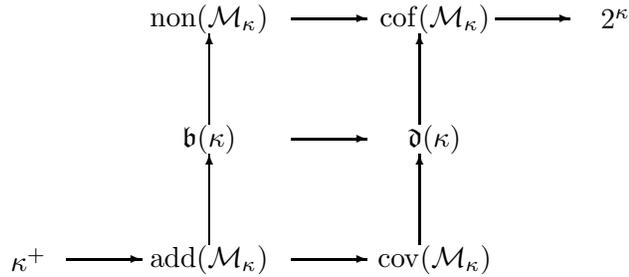
\begin{figure}[ht]
\begin{center}
\setlength{\unitlength}{0.2000mm}
\begin{picture}(530.0000,180.0000)(0,10)
\thinlines
\put(430,180){\vector(1,0){50}}
\put(295,180){\vector(1,0){50}}
\put(295,100){\vector(1,0){50}}
\put(295,20){\vector(1,0){50}}
\put(145,20){\vector(1,0){50}}
\put(380,30){\vector(0,1){60}}
\put(380,110){\vector(0,1){60}}
\put(240,110){\vector(0,1){60}}
\put(240,30){\vector(0,1){60}}
\put(470,170){\makebox(80,20){$2^\kappa$}}
\put(350,170){\makebox(80,20){$\cof(\mathcal{M}_\kappa)$}}
\put(350,90){\makebox(80,20){$\mathfrak{d}(\kappa)$}}
\put(350,10){\makebox(80,20){$\cov(\mathcal{M}_\kappa)$}}
\put(200,10){\makebox(80,20){$\add(\mathcal{M}_\kappa)$}}
\put(200,90){\makebox(80,20){$\mathfrak{b}(\kappa)$}}
\put(200,170){\makebox(80,20){$\non(\mathcal{M}_\kappa)$}}
\put(80,10){\makebox(80,20){$\kappa^+$}}

\end{picture}
\end{center}
\caption{Generalization of Cich\'{o}n's diagram (for $\kappa$ strongly inaccessible)}
\label{fig:cd}
\end{figure}

Also, the well known relationships between the classical cardinal invariants (See \cite{AB:CCC}) hold , namely:

\begin{lem}\label{maxmin} \hfill \\
$\add(\calM_\ka)=\min\{\frb(\ka),\cov(\calM_\ka)\}$ and $\cof(\calM_\ka)=\max\{\frd(\ka),\non(\calM_\ka)\}$.
\end{lem}

\section{Applications}

Until the end of the paper let $\kappa$, $\kappa^*$, $\Gamma$, $\alpha$ and $\mathbb{P}^*$ be fixed as in Theorem~\ref{thm_uc}.

\begin{thm}\label{thm_uc_simple} Let $G$ be $\PP^*$-generic. Then $V[G]$ satisfies
$\add(\calM_\ka)=\cof(\calM_\ka)=\non(\calM_\ka)=\cov(\calM_\ka)=
\frs(\ka)=\frr(\ka)=\frd(\ka)=\frb(\ka)=\ka^*$.
\end{thm}
\begin{proof}
Note that $\frb(\ka)\geq\ka^*$ because any set  of functions in $\ka^\ka$ of size $<\ka^\ast$ appears in some initial part of the iteration (by Lemma~\ref{smallinit}) and so is dominated by the Mathias generic functions added at later stages. On the other hand, any cofinal sequence of length $\kappa^*$ of the Mathias generics forms a dominating family. Thus $\mathfrak{d}(\kappa)\leq\kappa^*$ and since clearly $\mathfrak{b}(\kappa)\leq\mathfrak{d}(\kappa)$, we obtain $V^{\mathbb{P}^*}\vDash \mathfrak{b}(\kappa)=\mathfrak{d}(\kappa)=\kappa^*$.

To see that $\mathfrak{s}(\kappa)\geq\kappa^*$, observe that the Mathias generic subsets of $\kappa$ are unsplit and that every family of $\kappa$-reals of size $<\kappa$ is contained in $V^{\mathbb{P}_\beta}$ for some $\beta<\alpha$. On the other hand any cofinal sequence of length $\kappa^*$ of $\kappa$-Cohen reals forms a splitting family and so $V^{\mathbb{P}^*}\vDash \mathfrak{s}(\kappa)\leq\kappa^*$. Thus $V^{\mathbb{P}^*}\vDash \mathfrak{s}(\kappa)=\kappa^*$. That $\frr(\ka)=\ka^\ast$ follows from Proposition~\ref{mm}.

To verify the values of the characteristics associated to $\mathcal{M}_\kappa$, proceed as follows. Since $\mathfrak{b}(\ka) \leq \non(\calM_\ka)$,
$V^{\mathbb{P}^*}\vDash \kappa^*\leq\non(\calM_\ka)$. On the other hand any cofinal sequence of $\kappa$-Cohen reals of length $\kappa^*$ is a witness to $\non(\calM_\ka) \leq\ka^\ast$, since this set of $\kappa$-Cohen reals is non-meager. By a similar argument and the fact that $\mathfrak{d}(\kappa)=\kappa^*$ in $V^{\mathbb{P}^*}$, we obtain that $V^{\mathbb{P}^*}\vDash\cov(\calM_\ka)=\ka^*$. Now, Lemma~\ref{maxmin} implies that $\add(\calM_\ka)=\ka^*=\cof(\calM_\ka)$.
\end{proof}

Now, we are ready to prove our main theorem.

\begin{thm}\label{thm_main}
Suppose $\kappa$ is a supercompact cardinal, $\ka^\ast$ is a regular cardinal with
$\kappa <\kappa^{\ast} \leq \Gamma$ and $\Gamma$ satisfies $\Gamma^\kappa= \Gamma$. Then there is forcing extension in which cardinals have not been changed satisfying:
\begin{align*}
\ka^*&=\fru(\ka)=\frb(\ka)=\frd(\ka)=\fra(\ka)=\frs(\ka)=\frr(\ka)=\cov(\calM_\ka)\\
&=\add(\calM_\ka) =\non(\calM_\ka)=\cof(\calM_\ka) \;\hbox{and}\;2^\kappa=\Gamma.
\end{align*}
If in addition $\gamma<\kappa^*\rightarrow\gamma^{<\kappa}<\kappa^*$, then we can also provide that $\mathfrak{i}(\ka)= \ka^\ast$. If in addition $(\Gamma)^{<\kappa^*}\leq\Gamma$ then we can also provide that $\frp(\ka)=\frt(\ka)=\frh(\ka)=\kappa^*$.
\end{thm}
\begin{proof}
We will modify to iteration $\PP^*$ to an iteration $\bar{\PP}^*$ by specifying the iterands $\dot{\QQ}_j$ for every odd ordinal $j<\alpha$. It is easy to verify that those cardinal characteristics which were evaluated in the model of Theorem~\ref{thm_uc_simple} will have the same value $\kappa^*$ in $V^{\bar{\PP}}$. Let $\bar{\gamma}=\langle \gamma_i\ra_{i<\kappa^*}$ be a strictly increasing cofinal in $\alpha$ sequence of odd ordinals. The stages in $\bar{\gamma}$ will be used to add a $\kappa$-maximal almost disjoint family of size $\kappa^*$, as well as a $\kappa$-maximal independent family of size $\kappa^*$.

If $\Gamma^{<\kappa^*}\leq\Gamma$, then using an appropriate bookkeeping function $F$ with domain the odd ordinals in $\alpha$ which are not in the cofinal sequence $\bar{\gamma}$ we can use the generalized Mathias poset to add pseudointersections to all filter bases of size $<\kappa^*$ with the SIP. In case $\Gamma^{<\kappa^*}\not\leq\Gamma$, just take for odd stages which are not in $\bar{\gamma}$ arbitrary $\kappa$-centered, $\kappa$-directed closed forcing notions of size at most $\Gamma$.

To complete the definition of $\bar{\PP}^*$ it remains to specify the stages in $\bar{\gamma}$. For each $i<\kappa^*$, in $V^{\bar{\mathbb{P}}^*}_{\gamma_i}$ the poset $\dot{\QQ}_{\gamma_i}$ will be defined to be of the form $\QQ_{\gamma_i}=\QQ_{\gamma_i}^0*\dot{\QQ}_{\gamma_i}^1$. Fix a ground model $\kappa$-ad family $\mathcal{A}_0$ of size $\kappa$ and a ground model $\kappa$-independent family $\mathcal{I}_0$ of size $\kappa$. Let $\mathbb{Q}^0_{\gamma_0}=\bar{\mathbb{Q}}(\mathcal{A}_0,\kappa)$ (see Definition~\ref{mad_poset}) and in $V^{\bar{\PP}^*_{\gamma_0}*\dot{\QQ}_{\gamma_0}^0}$ let $\QQ^1_{\gamma_0}=\hat{\QQ}(\mathcal{I}_0,\kappa)$ (see Lemma~\ref{i}). Now,
fix any $i<\kappa^*$ and suppose that $\forall j<i$, $\QQ^0_{\gamma_j}=\bar{\QQ}(\mathcal{A}_j,\kappa)$ adds a generic subset $\bar{x}_{\gamma_j}$ of $\kappa$ where $\mathcal{A}_j=\mathcal{A}_0\cup\{\bar{x}_{\gamma_k}\}_{k<j}$ and that the poset $\QQ^1_{\gamma_j}=\hat{\QQ}(\mathcal{I}_j,\kappa)$ adds a subset $\hat{x}_{\gamma_j}$ of $\kappa$ where $\mathcal{I}_j=\mathcal{I}_0\cup\{\hat{x}_{\gamma_k}\}_{k<j}$. In $V^{\bar{\PP}^*_{\gamma_i}}$ let
$\QQ^0_{\gamma_i}=\bar{\QQ}(\mathcal{A}_i,\kappa)$ where $\mathcal{A}_i=\mathcal{A}_0\cup\{\bar{x}_{\gamma_j}\}_{j<i}$ and in $V^{\bar{\PP}^*_{\gamma_i}*\dot{\QQ}^0_{\gamma_i}}$ let $\QQ^1_{\gamma_i}=\hat{\QQ}(\mathcal{I}_i,\kappa)$ where $\mathcal{I}_i=\mathcal{I}_0\cup\{\hat{x}_{\gamma_j}\}_{j<i}$.

With this the recursive definition of the iteration $\bar{\PP}^*$ is defined. In $V^{\bar{\PP}^*}$ let $\mathcal{A}_*=\mathcal{A}_0\cup\{\bar{x}_{\gamma_j}\}_{j<\kappa^*}$ and let $\mathcal{I}_*=\mathcal{I}_0\cup\{\hat{x}_{\gamma_j}\}_{j<\kappa^*}$. We will show that $\mathcal{A}_*$ and $\mathcal{I}_*$ are a $\kappa$-mad and a $\kappa$-maximal independent families respectively. Clearly $\mathcal{A}_*$ is $\kappa$-ad and $\mathcal{I}_*$ is $\kappa$-independent. To show maximality of $\mathcal{A}_*$, consider an arbitrary $\bar{\PP}^*$-name $\dot{X}$ for a subset of $\kappa$ and suppose $\Vdash_{\bar{\PP}^*} (\{\dot{X}\}\cup\mathcal{A}_*\;\hbox{is }\kappa\hbox{-ad})$. By Lemma~\ref{smallinit} $\dot{X}$ can be viewed as a $\bar{\PP}^*_\beta$-name for some $\beta<\alpha$. Then for  $\gamma_j>\beta$, by Lemma~\ref{mad} we obtain $V^{\bar{\mathbb{P}^*}_\beta} \models \lvert \bar{x}_{\gamma_j}\cap\dot{X} \lvert = \kappa$, which is a contradiction.
Thus $\mathcal{A}_*$ is indeed maximal and so $\mathfrak{a}(\kappa)\leq\kappa^*$. However in $V^{\bar{\PP}^*}$, $\mathfrak{b}(\kappa)=\kappa^*$ and since $\mathfrak{b}(\kappa)\leq\mathfrak{a}(\kappa)$ we obtain $V^{\bar{\PP}^*}\vDash \mathfrak{a}(\kappa)=\kappa^*$.

To see that $\mathcal{I}_*$ is maximal, argue in a similar way. Consider arbitrary $\bar{\PP}^*$-name $\dot{X}$ for a subset of $\kappa$ such that $\Vdash_{\bar{\PP}^*} \{\dot{X}\}\cup\mathcal{I}_*\;\hbox{is independent}$. Then there is $\beta<\alpha$ such that we can see $\dot{X}$ as a $\bar{\PP}^*_\beta$-name. Let $\gamma_j>\beta$. Then by Lemma~\ref{i}, in $V^{\bar{\PP}^*_{\gamma_j+1}}$ the family $\{\hat{x}_{\gamma_j}\}\cup \mathcal{I}_{\gamma_j}\cup\{X\}$ is not independent, which is a contradiction. Thus $\mathcal{I}_*$ is maximal and so $\mathfrak{i}(\kappa)\leq\kappa^*$. On the other hand if whenever $\gamma<\kappa^*$ we have $\gamma^{<\kappa}<\kappa^*$, then $\mathfrak{d}(\kappa)\leq\mathfrak{i}(\kappa)$ (by Lemma~\ref{indep}) and since in $V^{\bar{\PP}^*}$,  $\mathfrak{d}(\kappa)=\kappa^*$, we obtain $V^{\bar{\PP}^*}\vDash \mathfrak{i}(\kappa)=\kappa^*$.

Suppose $\Gamma^{<\kappa^*}\leq\Gamma$. In this case, every filter of size $<\kappa^*$ with the SIP has a pseudointersection in $V^{\bar{\PP}^*}$. Thus in the final extension $\mathfrak{p}(\kappa)\geq\kappa^*$. However $\mathfrak{p}(\kappa)\leq\mathfrak{t}(\kappa)\leq\mathfrak{s}(\kappa)$ and since $V^{\bar{\PP}^*}\vDash \mathfrak{s}(\kappa)=\kappa^*$, we obtain that $\mathfrak{p}(\kappa)=\mathfrak{t}(\kappa)=\kappa^\ast$. By Proposition~\ref{prop31}, $\mathfrak{h}(\kappa)\leq\mathfrak{s}(\kappa)=\kappa^\ast$ and $\kappa^\ast=\mathfrak{t}(\kappa)\leq\mathfrak{h}(\kappa)$. Thus $\mathfrak{h}(\kappa)=\kappa^\ast$.
\end{proof}

The above iteration can be additionally modified so that in the final extension the minimal size of a $\kappa$-maximal cofinitary group, $\mathfrak{a}_g(\kappa)$, is $\kappa^\ast$. Indeed, one can use the stages in $\bar{\gamma}$ and \cite[Definition 2.2.]{VF:MCGR} to add a $\kappa$-maximal cofinitary group of size $\kappa^\ast$. The fact that $\kappa^\ast\leq\mathfrak{a}_g(\kappa)$ follows from $\mathfrak{b}(\kappa)\leq\mathfrak{a}_g(\kappa)$ (see~\cite{BHZ:MFN}).\\

{\emph{Acknowledgements}}:
The first author would like to thank for their generous support both
the Japan Society for the Promotion of Science (JSPS) through
a JSPS Postdoctoral Fellowship for
Foreign Researchers and
JSPS Grant-in-Aid 23 01765, and the UK Engineering and Physical Sciences
Research Council through an EPSRC Early Career Fellowship,
reference EP/K035703/1.
The last three authors would like to thank the Austrian Science Fund (FWF) for their generous support through grants number
M1365 - N13 (V. Fischer) and FWF Project P25748 (S. D. Friedman and D. Montoya).





\end{document}